\documentclass[12pt]{amsart}
\usepackage{amsmath,amssymb,amsthm,bm,bbm}
\usepackage{graphicx,enumerate}
\usepackage{layout}

\theoremstyle{plain}
\newtheorem{theorem}{Theorem}[section]
\newtheorem{lemma}[theorem]{Lemma}

\theoremstyle{definition}
\newtheorem{definition}[theorem]{Definition}
\newtheorem{example}[theorem]{Example}
\newtheorem{remark}[theorem]{Remark}

\newcommand{\bZ}{\mathbbm{Z}}\newcommand{\bQ}{\mathbbm{Q}}
\newcommand{\bC}{\mathbbm{C}}
\newcommand{\Br}{\mathrm{Br}}
\newcommand{\of}{\overline{f}}
\newcommand{\oh}{\overline{h}}

\numberwithin{equation}{section}

\title[Unramified Brauer Groups for Groups of Order $p^5$]
{\Large Unramified Brauer Groups\vspace*{2mm}\\ for Groups of Order $p^5$}

\begin{document}
\maketitle
\begin{center}
\begin{tabular}{lll}
Akinari Hoshi& &Ming-chang Kang\\
Department of Mathematics & &Department of Mathematics and\\
Rikkyo University & and\hspace*{3mm} &Taida Institute of Mathematical Sciences\\
Tokyo, Japan & &National Taiwan University\\
E-mail: \texttt{hoshi@rikkyo.ac.jp}&  &Taipei, Taiwan\\
 & & E-mail: \texttt{kang@math.ntu.edu.tw}
\end{tabular}
\end{center}

\renewcommand{\thefootnote}{\fnsymbol{footnote}}

\footnote[0]{


The first-named author was partially supported by KAKENHI (22740028).
Both authors were partially supported by National Center for
Theoretic Sciences (Taipei office). The work of this paper was
finished while the first-named author visited National Center for
Theoretic Sciences (Taipei).}

\renewcommand{\thefootnote}{\arabic{footnote}}

\bigskip
\noindent Abstract. Let $k$ be any field, $G$ be a finite group
acting on the rational function field $k(x_g : g\in G)$ by $h\cdot
x_g=x_{hg}$ for any $h,g\in G$. Define $k(G)=k(x_g : g\in G)^G$.
Noether's problem asks whether $k(G)$ is rational (= purely
transcendental) over $k$. It is known that, if $\bC(G)$ is
rational over $\bC$, then $B_0(G)=0$ where $B_0(G)$ is the
unramified Brauer group of $\bC(G)$ over $\bC$. Bogomolov showed
that, if $G$ is a $p$-group of order $p^5$, then $B_0(G)=0$. This
result was disproved by Moravec for $p=3,5,7$ by computer
computing. We will give a theoretic proof of the following theorem
(i.e. by the traditional bare-hand proof without using computers).
 Theorem. Let $p$ be any odd prime number. Then there is a group
$G$ of order $p^5$ satisfying $B_0(G)\neq 0$ and $G/[G,G] \simeq
C_p \times C_p$. In particular, $\bC(G)$ is not rational over
$\bC$.

\section{Introduction}\label{sec-intro}

Let $k$ be any field and $G$ be a finite group.
Let $G$ act on the rational function field $k(x_g : g\in G)$ by
$k$-automorphisms such that $g\cdot x_h=x_{gh}$ for any $g,h\in G$.
Denote by $k(G)$ the fixed field $k(x_g : g\in G)^G$.
Noether's problem asks whether $k(G)$ is rational (= purely transcendental) over $k$.
It is related to the inverse Galois problem,
to the existence of generic $G$-Galois extensions over $k$,
and to the existence of versal $G$-torsors over $k$-rational field extensions
\cite[33.1, page 86]{Sw, Sa1, GMS}.
Noether's problem for abelian groups was studied extensively by Swan, Voskresenskii,
Endo, Miyata and Lenstra, etc.
The reader is referred to Swan's paper for a survey of this problem \cite{Sw}.

On the other hand, just a handful of results about Noether's problem are obtained
when the groups are not abelian.
It is the case even when the group $G$ is a $p$-group.

Before stating the results of Noether's problem for non-abelian $p$-groups,
we recall some relevant definitions.

\begin{definition}
Let $k\subset K$ be an extension of fields. $K$ is rational over
$k$ (for short, $k$-rational) if $K$ is purely transcendental over
$k$. $K$ is stably $k$-rational if $K(y_1,\ldots,y_m)$ is rational
over $k$ for some $y_1,\ldots,y_m$ such that $y_1,\ldots,y_m$ are
algebraically independent over $K$. When $k$ is an infinite field,
$K$ is said to be retract $k$-rational if there is a $k$-algebra
$A$ contained in $K$ such that (i) $K$ is the quotient field of
$A$, (ii) there is some non-zero polynomial $f\in
k[X_1,\ldots,X_n]$, the polynomial ring over $k$, and there are
$k$-algebra homomorphisms $\varphi : A\rightarrow
k[X_1,\ldots,X_n][1/f]$ and $\psi :
k[X_1,\ldots,X_n][1/f]\rightarrow A$ satisfying
$\psi\circ\varphi=1_A$. (See \cite{Sa2, Ka} for details.) It is
not difficult to see that ``$k$-raional''$\Rightarrow$ ``stably
$k$-rational''$\Rightarrow$ ``retract $k$-rational''.
\end{definition}
\begin{definition}
Let $k\subset K$ be an extension of fields. The notion of the
unramified Brauer group of $K$ over $k$, denoted by $\Br_{v,k}(K)$
was introduced by Saltman \cite{Sa3}. By definition,
$\Br_{v,k}(K)=\bigcap_R {\rm Image} \{ \Br(R)\rightarrow \Br(K)
\}$ where $\Br(R)\rightarrow \Br(K)$ is the natural map of Brauer
groups and $R$ runs over all the discrete valuation rings $R$ such
that $k\subset R\subset K$ and $K$ is the quotient field of $R$.
\end{definition}
\begin{lemma}[Saltman {\cite{Sa3, Sa4}}]\label{lemSa}
If $k$ is an infinite field and $K$ is retract $k$-rational, then the natural
map $\Br(k)\rightarrow \Br_{v,k}(K)$ is an isomorphism.
In particular, if $k$ is an algebraically closed field and $K$ is retract
$k$-rational, then $\Br_{v,k}(K)=0$.
\end{lemma}
\begin{theorem}[{Bogomolov, Saltman \cite[Theorem 12]{Bo, Sa4}}]\label{thBS}
Let $G$ be a finite group, $k$ be an algebraically closed field with
$\gcd\{|G|,{\rm char}\, k\}=1$.
Let $\mu$ denote the multiplicative subgroup of all roots of unity in $k$.
Then $\Br_{v,k}(k(G))$ is isomorphic to the group $B_0(G)$ defined by
\[
B_0(G)=\bigcap_A {\rm Ker}\{{\rm res}^A_G : H^2(G,\mu)\rightarrow H^2(A,\mu)\}
\]
where $A$ runs over all the bicyclic subgroups of $G$ $($a group $A$ is called
bicyclic if $A$ is either a cyclic group or a direct product of two cyclic
groups$)$.
\end{theorem}
Following Kunyavskii \cite{Ku} we will call $B_0(G)$ the Bogomolov
multiplier of $G$. Because of Theorem \ref{thBS} we will not
distinguish $B_0(G)$ and $\Br_{v,k}(k(G))$ when $k$ is
algebraically closed and $\gcd\{|G|,{\rm char}\, k\}=1$. In this
situation, $B_0(G)$ is canonically isomorphic to $\bigcap_A {\rm
Ker}\{{\rm res}^A_G : H^2(G, \bQ/\bZ)\rightarrow H^2(A,
\bQ/\bZ)\}$, i.e. we may replace the coefficient $\mu$ by
$\bQ/\bZ$ in Theorem \ref{thBS}.

Using the unramified Brauer groups, Saltman and Bogomolov are able to
establish counter-examples to Noether's problem for non-abelian $p$-groups.

\begin{theorem}\label{thSB}
Let $p$ be any prime number, $k$ be any algebraically closed field with
char $k\neq p$.

{\rm (1) (Saltman}\ \cite{Sa3}{\rm )} There is a group $G$ with
order $p^9$ such that $B_0(G)\neq 0$. In particular, $k(G)$ is not
retract $k$-rational. Thus $k(G)$ is not $k$-rational.

{\rm (2) (Bogomolov}\ \cite{Bo}{\rm )} There is a group $G$ with
order $p^6$ such that $B_0(G)\neq 0$. Thus $k(G)$ is not
$k$-rational.
\end{theorem}
For $p$-groups of small order, we have the following result.
\begin{theorem}[{Chu and Kang \cite{CK}}]\label{thCK}
Let $p$ be any prime number, $G$ is a $p$-group of order $\leq p^4$
and of exponent $e$.
If $k$ is a field satisfying either
{\rm (i)} {\rm char} $k=p$, or {\rm (ii)} $k$ contains a primitive
$e$-th root of unity, then $k(G)$ is $k$-rational.
\end{theorem}
Because of the above Theorem \ref{thSB} and Theorem \ref{thCK},
we may wonder what happens to non-abelian $p$-groups of order $p^5$.
\begin{theorem}[{Chu, Hu, Kang and Prokhorov \cite{CHKP}}]
Let $G$ be a group of order $32$ and of exponent $e$.
If $k$ is a field satisfying either {\rm (i)} {\rm char} $k=2$,
or {\rm (ii)} $k$ contains a primitive $e$-th root of unity,
then $k(G)$ is $k$-rational.
\end{theorem}
\begin{theorem}\label{thBBMP}
{\rm (1)} \cite[Lemma 4.11]{Bo} If $G$ is a $p$-group with
$B_0(G)\neq 0$ and $G/[G,G]\cong C_p\times C_p$, then $p\geq 5$
and $|G|\geq p^7$.

{\rm (2)} \cite[Corollary 2.11]{BMP} If $G$ is a $p$-group of order $\leq p^5$,
then $B_0(G)=0$.
\end{theorem}
It came as a surprise that Moravec's recent paper \cite{Mo} disproved
the above Theorem \ref{thBBMP}.
\begin{theorem}[{Moravec \cite[Section 5]{Mo}}]\label{thMo}
Let $G=G(243,i)$, $28\leq i\leq 30$, where $G(243,i)$ is the $i$-th group of
groups of order $243$ in the database of GAP.
Then $B_0(G)\neq 0$.
Moreover, if $G$ is a group of order $243$ other than $G(243,i)$ with
$28\leq i\leq 30$, then $B_0(G)=0$.
\end{theorem}
Moravec proves Theorem \ref{thMo} by using computer computing. No
theoretic proof is given. A file of the GAP functions and commands
for computing $B_0(G)$ can be found at Moravec's website
\verb+www.fmf.uni-lj.si/~moravec/b0g.g+. More recently, Moravec
was able to classify all groups $G$ with order $p^5$ (when $p=5$
and $p=7$) such that $B_0(G)\neq 0$ by using computers again.

The main result of this paper is the following theorem.
\begin{theorem}\label{thmain}
Let $p$ be any odd prime number. Then there is a $p$-group $G$ of
order $p^5$ such that $B_0(G)\neq 0$ and $G/[G,G] \simeq C_p
\times C_p$. In particular, if $k$ is a field with {\rm char}
$k\neq p$, then $k(G)$ is not retract $k$-rational. Thus $k(G)$ is
not $k$-rational.
\end{theorem}
As a corollary of the above theorem, we record the following result.
\begin{theorem}\label{thre}
Let $n$ be a positive integer and $k$ be a field with
$\gcd\{|G|,{\rm char}\, k\}=1$. If $2^6\mid n$ or $p^5\mid n$ for
some odd prime number $p$, then there is a group $G$ of order $n$
such that $B_0(G)\neq 0$. In particular, $k(G)$ is not stably
$k$-rational; when $k$ is an infinite field, $k(G)$ is not even
retract $k$-rational.
\end{theorem}

In Section \ref{seMore} we will produce more groups $G$ with order
$p^5$ and $B_0(G)\neq 0$ for $p \geq 5$. In particular, we find
$6$ such groups for $p=5$ and $7$; these groups are exactly those
groups $G$ with non-trivial $B_0(G)$ obtained by Moravec using his
computer program.

Finally we remark that recently Chu, Hu and Kang prove the
following theorem : Let $G$ be a group of order $243$ and of
exponent $e$. If $k$ is a field containing a primitive $e$-th root
of unity and $G$ is not isomorphic to $G(243,i)$ for $28\leq i\leq
30$, then $k(G)$ is $k$-rational \cite{CHK}.

The proof of Theorem \ref{thmain} is divided into two parts, $p=3$
and $p\geq 5$. According to the computation of Moravec, the groups
$G$ of order $p^5$ with $B_0(G)\neq 0$ for $p=5$ and $p=7$ look
very similar. But their outlooks are not similar to those of
$G(3^5,i)$ for $28\leq i\leq 30$. Thus for $p\geq 5$ we define a
group $G$ of order $p^5$ in a uniform way, i.e. independent of the
value of the prime number $p$. Then we try to prove $B_0(G)\neq 0$
for this group $G$ (and also for $G=G(3^5,i), 28\leq i\leq 30$).

The idea of proving $B_0(G)\neq 0$ goes as follows. Take a
suitable normal subgroup $N$ of $G$. Consider the $5$-term exact
sequence of Hochschild and Serre \cite{HS},
\[
0\rightarrow H^1(G/N,\bQ/\bZ)\rightarrow
H^1(G,\bQ/\bZ) \rightarrow
H^1(N,\bQ/\bZ)^G\rightarrow
H^2(G/N,\bQ/\bZ)\stackrel{\psi}{\rightarrow}H^2(G,\bQ/\bZ)
\]
where $\psi$ is the inflation map. We will show that the image of
$\psi$ is non-zero and is contained in $B_0(G)$. Thus $B_0(G)$ is
non-trivial.

The above method can be applied to other groups of order $p^5$ if $p\geq 5$.
In particular, if $p\equiv 1\pmod{12}$, we find at least $8$ distinct groups
$G$ with $B_0(G)\neq 0$; when $p\equiv 5$ or $7 \pmod{12}$, at least $6$ such groups;
when $p\equiv 11\pmod{12}$, at least $4$ such groups.
See Section \ref{seMore} for details.

The paper is organized as follows. In Section \ref{setwolem}, two
lemmas are proved. The proof of Theorem \ref{thmain} is given in
Section \ref{seproof}. Theorem \ref{thre} follows as an
application of Theorem \ref{thmain} and Theorem \ref{thSB}. We
emphasize that the proof of Theorem \ref{thmain} does not rely on
computers; only the bare hands are sufficient to finish the proof
of Theorem \ref{thmain}, contrasting with Moravec's
computer-relying proof which is valid only for small prime numbers
\cite{Mo}. In Section \ref{seMore}, we prove more groups $G$ with
$B_0(G)\neq 0$ when $p\equiv 1,5,7,11\pmod{12}$.

Standing notations. Throughout this paper, $k$ is a field,
$\zeta_n$ denotes a primitive $n$-th root of unity. Whenever we
write $\zeta_n\in k$, it is understood that either char $k=0$ or
char $k=l>0$ with $l\,{\not{\mid}}\ n$. When $k$ is an
algebraically closed field, $\mu$ denotes the set of all roots of
unity, i.e. $\mu=\{\alpha\in k\setminus\{0\} : \alpha^n=1\ {\rm
for\ some\ integer} \ n\ {\rm depending\ on}\ \alpha\}$. If $G$ is
a group, $Z(G)$ and $[G,G]$ denote the center and the commutator
subgroup of the group $G$ respectively. The exponent of a group
$G$ is defined as ${\rm lcm}\{{\rm ord}(g) : g\in G\}$ where ${\rm
ord}(g)$ is the order of the element $g$. We denote by $C_n$ the
cyclic group of order $n$. A group $G$ is called a bicyclic group
if it is either a cyclic group or a direct product of two cyclic
groups. When we write cohomology groups $H^q(G,\mu)$ or
$H^q(G,\bQ/\bZ)$, it is understood that $\mu$ and $\bQ/\bZ$ are
trivial $G$-modules.

For emphasis, recall the definition of $k(G)$ which was defined in
the first paragraph of this section. The group $G(n,i)$ is the
$i$-th group among the groups of order $n$ in GAP. The version of
GAP we refer to in this paper is GAP4, Version: 4.4.12 \cite{GAP}.

\bigskip

\section{Two lemmas}\label{setwolem}

Throughout this paper, when $G$ is a group, $g,h\in G$, we will denote by
$[g,h]$ the element $g^{-1}h^{-1}gh$.
When $N$ is a normal subgroup of $G$ and $g\in G$, the element $\overline{g}\in G/N$
denotes the image of $g$ in the quotient group $G/N$.
\begin{lemma}\label{lem21}
Let $G$ be a finite group, $N$ be a normal subgroup of $G$.
Assume that {\rm (i)} ${\rm tr} : H^1(N,\bQ/\bZ)^G\rightarrow H^2(G/N,\bQ/\bZ)$
is not surjective where {\rm tr} is the transgression map, and {\rm (ii)} for any
bicyclic subgroup $A$ of $G$, the group $AN/N$ is a cyclic subgroup of $G/N$.
Then $B_0(G)\neq 0$.
\end{lemma}
\begin{proof}
Consider the Hochschild-Serre $5$-term exact sequence
\[
0\rightarrow H^1(G/N,\bQ/\bZ)\rightarrow
H^1(G,\bQ/\bZ) \rightarrow
H^1(N,\bQ/\bZ)^G\stackrel{\rm tr}{\rightarrow}
H^2(G/N,\bQ/\bZ)\stackrel{\psi}{\rightarrow}H^2(G,\bQ/\bZ)
\]
where $\psi$ is the inflation map \cite{HS}.

Since {\rm tr} is not surjective, we find that $\psi$ is not the zero map.
Thus ${\rm Image}(\psi)\neq 0$.

We will show that ${\rm Image}(\psi)\subset B_0(G)$. By
definition, it suffices to show that, for any bicyclic subgroup
$A$ of $G$, the composite map
$H^2(G/N,\bQ/\bZ)\stackrel{\psi}{\rightarrow}H^2(G,\bQ/\bZ)
\stackrel{\rm res}{\rightarrow}H^2(A,\bQ/\bZ)$ becomes the zero
map where {\rm res} is the restriction map. Consider the following
commutative diagram
\begin{align*}
H^2(G/N,\bQ/\bZ)\stackrel{\psi}{\rightarrow} H^2(G,\bQ/\bZ)
\stackrel{\rm res}{\rightarrow} H^2(A,\bQ/\bZ)\\
\psi_0\downarrow\hspace*{5.5cm}\uparrow\psi_1\hspace*{1.0cm}\\
H^2(AN/N,\bQ/\bZ)\ \ \ \ \
\stackrel{\widetilde{\psi}}{\cong}\ \ \ \ \
H^2(A/A\cap N,\bQ/\bZ)
\end{align*}
where $\psi_0$ is the restriction map,
$\psi_1$ is the inflation map,
$\widetilde{\psi}$ is the natural isomorphism.

Since $AN/N$ is cyclic, write $AN/N\cong C_m$ for some integer $m$.
It is well-known that $H^2(C_m,\bQ/\bZ)=0$ (see, for example,
\cite[page 37, Corollary 2.2.12]{Kar}).
Hence $\psi_0$ is the zero map.
Thus ${\rm res}\circ \psi : H^2(G/N,\bQ/\bZ)\rightarrow H^2(A,\bQ/\bZ)$
is also the zero map.

As ${\rm Image}(\psi)\subset B_0(G)$ and ${\rm Image}(\psi)\neq 0$,
we find that $B_0(G)\neq 0$.
\end{proof}
\begin{lemma}\label{lemf}
Let $p\geq 3$ and $G$ be a $p$-group of order $p^5$ generated by
$f_i$ where $1\leq i\leq 5$.
Suppose that, besides other relations, the generators $f_i$'s satisfy
the following conditions.

{\rm (i)} $f_4^p=f_5^p=1$, $f_5\in Z(G)$,

{\rm (ii)} $[f_2,f_1]=f_3, [f_3,f_1]=f_4, [f_4,f_1]=[f_3,f_2]=f_5,
[f_4,f_2]=[f_4,f_3]=1$,

and

{\rm (iii)} $\langle f_4, f_5\rangle\cong C_p\times C_p$,
$G/\langle f_4,f_5\rangle$ is a non-abelian group of order $p^3$
and of exponent $p$.

Then $B_0(G)\neq 0$.
\end{lemma}
\begin{remark}
When $p=2$ and $G/N$ is a non-abelian group of order $8$,
then $H^2(G/N,\bQ/\bZ)$ $=0$ or $C_2$ \cite[page 138, Theorem 3.3.6]{Kar}.
Thus {\rm tr} : $H^1(N,\bQ/\bZ)^G\rightarrow H^2(G/N,\bQ/\bZ)$ in
Lemma \ref{lemf} may become surjective. This is the reason why we assume
$p\geq 3$ in this lemma.
\end{remark}
\begin{proof}
Choose $N=\langle f_4,f_5\rangle$. We will check the conditions in
Lemma \ref{lem21} are satisfied. Thus $B_0(G)\neq 0$.

Step 1.
Since $N\cong C_p\times C_p$, we find that $H^1(N,\bQ/\bZ)\cong C_p\times C_p$.

Define $\varphi_1, \varphi_2\in H^1(N,\bQ/\bZ)={\rm Hom}(N,\bQ/\bZ)$ by
$\varphi_1(f_4)=1/p$, $\varphi_1(f_5)=0$,
$\varphi_2(f_4)=0$, $\varphi_2(f_5)=1/p$.
Clearly $H^1(N,\bQ/\bZ)=\langle\varphi_1,\varphi_2\rangle$.

The action of $G$ on $\varphi_1, \varphi_2$ are given by
${}^{f_1}\varphi_1(f_4)=\varphi_1(f_1^{-1}f_4f_1)
=\varphi(f_4f_5)=\varphi_1(f_4)+\varphi_1(f_5)=1/p$,
${}^{f_1}\varphi_1(f_5)=\varphi_1(f_1^{-1}f_5f_1)
=\varphi(f_5)=0$. Thus ${}^{f_1}\varphi_1=\varphi_1$. Similarly,
${}^{f_1}\varphi_2(f_4)=1/p$, ${}^{f_1}\varphi_2(f_5)=1/p$ and
${}^{f_1}\varphi_2=\varphi_1+\varphi_2$.

For any $\varphi\in
H^1(N,\bQ/\bZ)=\langle\varphi_1,\varphi_2\rangle\cong C_p\times
C_p$, write $\varphi=a_1\varphi_1+a_2\varphi_2$ for some integers
$a_1,a_2\in\bZ$ (modulo $p$). Since
${}^{f_1}\varphi={}^{f_1}(a_1\varphi_1+a_2\varphi_2)
=a_1(^{f_1}\varphi_1)+a_2({}^{f_1}\varphi_2)=(a_1+a_2)\varphi_1+a_2\varphi_2$,
we find that ${}^{f_1}\varphi=\varphi$ if and only if $a_2=0$,
i.e. $\varphi\in\langle\varphi_1\rangle$. On the other hand, it is
easy to see that ${}^{f_2}\varphi_1=\varphi_1={}^{f_3}\varphi_1$
and therefore $\varphi_1\in H^1(N,\bQ,\bZ)^G$. We find
$H^1(N,\bQ/\bZ)^G=\langle\varphi_1\rangle\cong C_p$.

By \cite[Proposition 6.3; Kar, page 138, Theorem 3.3.6]{Le}, since
$G/N$ is a non-abelian group of order $p^3$ and of exponent $p$,
we find $H^2(G/N,\bQ/\bZ)\cong C_p\times C_p$. Thus ${\rm tr} :
H^1(N,\bQ/\bZ)^G\rightarrow H^2(G/N,\bQ/\bZ)$ is not surjective.
Hence the first condition of Lemma \ref{lem21} is verified.

Step 2. We will verify the second condition of Lemma \ref{lem21},
i.e. for any bicyclic subgroup $A$ of $G$, $AN/N$ is a cyclic
group.

Before the proof, we list the following formulae which are
consequences of the commutator relations, i.e. relations {\rm
(ii)} of this lemma. The proof of these formulae is routine and is
omitted.

For $1\leq i,j\leq p-1$,
$f_4^if_1^j=f_1^jf_4^if_5^{ij}$,
$f_3^if_2^j=f_2^jf_3^if_5^{ij}$,

$f_3^if_1^j=f_1^jf_3^if_4^{ij}f_5^{i\cdot \binom{j}{2}}$,

$f_2^if_1^j=f_1^jf_2^if_3^{ij}f_4^{i\cdot\binom{j}{2}}
f_5^{i\cdot\binom{j}{3}+\binom{i}{2}\cdot j}$\\
where $\binom{a}{b}$ denotes the binomial coefficient when $a\geq b\geq 1$ and
we adopt the convention $\binom{a}{b}=0$ if $1\leq a<b$.

Moreover, in $G/N$, $(\of_1^j\of_2^i)^e
=\of_1^{ej}\of_2^{ei}\of_3^{\binom{e}{2}\cdot ij}$ for $1\leq
i,j\leq p-1$, $1\leq e\leq p$.

Step 3. Let $A=\langle h_1,h_2\rangle$ be a bicyclic subgroup of $G$.
We will show that $AN/N$ is cyclic in $G/N$.

Since $AN/N$ is abelian and $G/N$ is not abelian.
We find that $AN/N$ is a proper subgroup of $G/N$ which is of order $p^3$.

If $|AN/N|\leq p$, then $AN/N$ is cyclic. From now on, we will
assume $AN/N$ is an order $p^2$ subgroup and try to find a
contradiction.

In $G/N$, write $\oh_1=\of_1^{a_1}\of_2^{a_2}\of_3^{a_3}$,
$\oh_2=\of_1^{b_1}\of_2^{b_2}\of_3^{b_3}$ for some integers $a_j$,
$b_j$ (recall that $G/N=\langle \of_1,\of_2,\of_3\rangle$ and
$A=\langle h_1,h_2\rangle$). After suitably changing the
generators $h_1$ and $h_2$, we will show that there are only three
possibilities: $(\oh_1, \oh_2)=(\of_2, \of_3)$,
$(\of_1\of_3^{a_3}, \of_2\of_3^{b_3})$, $(\of_1\of_2^{a_2},
\of_3)$ for some integers $a_2,a_3,b_3$.

Suppose $\oh_1=\of_1^{a_1}\of_2^{a_2}\of_3^{a_3}$ and
$\oh_2=\of_1^{b_1}\of_2^{b_2}\of_3^{b_3}$ as above. If
$a_1=b_1=0$, then $\langle
\oh_1,\oh_2\rangle=\langle\of_2,\of_3\rangle$. Thus after changing
the generating elements $h_1,h_2$, we may assume that
$\oh_1=\of_2, \oh_2=\of_3$. This is the first possibility.

If $a_1\not\equiv 0$ or $b_1\not\equiv 0\pmod{p}$, we may assume
$1\leq a_1\leq p-1$. Find an integer $e$ such that $1\leq e\leq
p-1$ and $a_1e\equiv 1\pmod{p}$. Use the formulae in Step 2, we
get $\oh_1^e=\of_1\of_2^{c_2}\of_3^{c_3}$. Since $\langle
h_1,h_2\rangle=\langle h_1^e,h_2\rangle$, without loss of
generality, we may assume that $\oh_1=\of_1\of_2^{a_2}\of_3^{a_3}$
(i.e. $a_1=1$ from the beginning).

Since $\langle h_1,h_2\rangle=\langle h_1,(h_1^{b_1})^{-1}h_2\rangle$, we may
assume $\oh_1=\of_1\of_2^{a_2}\of_3^{a_3}$ and $\oh_2=\of_2^{b_2}\of_3^{b_3}$.

In case $1\leq b_2\leq p-1$, take an integer $e'$ with $1\leq
e'\leq p-1$ and $b_2 e'\equiv 1\pmod{p}$. Use the generating set
$\langle h_1,h_2^{e'}\rangle$ for A. Thus we may assume
$\oh_1=\of_1\of_3^{a_3}$, $\oh_2=\of_2\of_3^{b_3}$. This is the
second possibility.

If $b_2\equiv 0\pmod{p}$, then $\oh_1=\of_1\of_2^{a_2}\of_3^{a_3}$,
$\oh_2=\of_3^{b_3}$.
If $b_3=0$, then $AN/N$ is cyclic.
Thus $b_3\not\equiv 0\pmod{p}$.
Changing the generators again, we may assume
$\oh_1=\of_1\of_2^{a_2}$, $\oh_2=\of_3$.
This is the third possibility.

Step 4.
We will show that all three possibilities in Step 3 lead to contradiction.

Suppose $\oh_1=\of_2, \oh_2=\of_3$. Write
$h_1=f_2f_4^{a_4}f_5^{a_5}$, $h_2=f_3f_4^{b_4}f_5^{b_5}$. Since
$h_1h_2=h_2h_1$, we get
$f_2f_4^{a_4}f_3f_4^{b_4}=f_3f_4^{b_4}f_2f_4^{a_4}$ (because
$f_5\in Z(G)$). Rewrite this identity with the help of the
formulae in Step 2. We get
$f_2f_3f_4^{a_4+b_4}=f_2f_3f_4^{a_4+b_4}f_5$, which is a
contradiction.

Suppose $\oh_1=\of_1\of_3^{a_3}$, $\oh_2=\of_2\of_3^{b_3}$.
In $G/N$, we have $\oh_1\oh_2=\oh_2\oh_1$.
But it is obvious the two elements $\of_1\of_3^{a_3}$, $\of_2\of_3^{b_3}$
do not commute. Done.

Suppose $\oh_1=\of_1\of_2^{a_2}$, $\oh_2=\of_3$.
Write $h_1=f_1f_2^{a_2}f_4^{a_4}f_5^{a_5}$,
$h_2=f_3f_4^{b_4}f_5^{b_5}$.
Use the fact $h_1h_2=h_2h_1$.
It is easy to find a contradiction.
\end{proof}

\bigskip

\section{Proof of Theorem \ref{thmain} and Theorem \ref{thre}}\label{seproof}

Proof of Theorem \ref{thmain} ---------

First we will show that, if $G$ is a $p$-group with $B_0(G)\neq 0$
and $k$ is any field with char $k\neq p$, then $k(G)$ is not
retract $k$-rational.

Suppose not. Assume that $k(G)$ is retract $k$-rational. Then
$\overline{k}(G)$ is also retract $\overline{k}$-rational where
$\overline{k}$ is the algebraic closure of $k$. By Lemma
\ref{lemSa} and Theorem \ref{thBS}, we obtain $B_0(G)=0$, which is
a contradiction.

It remains to show that there is a group $G$ of order $p^5$ with $B_0(G)\neq 0$.

If $p=3$, define $G=G(243,28)$. Use the database of GAP. The
generators and relations of $G$ are given by

$G=\langle f_i : 1\leq i\leq 5\rangle$, $Z(G)=\langle f_5\rangle$,
and the relations

$f_1^3=f_4^3=f_5^3=1$, $f_2^3=f_4^{-1}$, $f_3^3=f_5^{-1}$,

$[f_2,f_1]=f_3, [f_3,f_1]=f_4, [f_4,f_1]=[f_3,f_2]=f_5,
[f_4,f_2]=[f_4,f_3]=1$.

Note that this group $G$ satisfies the conditions of Lemma
\ref{lemf}. Apply Lemma \ref{lemf}. We find $B_0(G)\neq 0$. It is
not difficult to see that $[G,G] = \langle f_3,f_4,f_5\rangle$ and
$G/[G,G] \simeq C_p \times C_p$.

Suppose $p\geq 5$. Define a group $G$ by generators and relations
as follows,

$G=\langle f_i : 1\leq i\leq 5\rangle$, $Z(G)=\langle f_5\rangle$
with the additional relations

$f_1^p=f_5$, $f_i^p=1$ for $2\leq i\leq 5$,

$[f_2,f_1]=f_3, [f_3,f_1]=f_4, [f_4,f_1]=[f_3,f_2]=f_5,
[f_4,f_2]=[f_4,f_3]=1$.

This is a well-defined group of order $p^5$ by Bender's
classification \cite{Be}. For, this group is the group 43 with
$n=0$ in \cite[page 69]{Be}. This group is the group $G_0(2|p)$
defined in the next section.

Apply Lemma \ref{lemf} again. We find that $B_0(G)\neq 0$.
Moreover, $G/[G,G] \simeq C_p \times C_p$. Done.

\bigskip
Proof of Theorem \ref{thre} ---------

Suppose that $p^5\mid n$ for some odd prime number $p$. Write
$n=p^5m$. By Theorem \ref{thmain} choose a group $G_0$ of order
$p^5$ satisfying $B_0(G_0)\neq 0$. Define $G=G_0\times C_m$.

We will prove that $k(G)$ is not stably $k$-rational (resp. not
retract $k$-rational if $k$ is infinite). Suppose not. Assume that
$k(G)$ is stably $k$-rational (resp. retract $k$-rational if $k$
is infinite) Then so is $\overline{k}(G)$ over $\overline{k}$
where $\overline{k}$ is the algebraic closure of $k$. In
particular, $\overline{k}(G)$ is retract $\overline{k}$-rational.
Since $G=G_0\times C_m$, by \cite[Theorem 1.5, Ka, Lemma
3.4]{Sa1}, we find that $\overline{k}(G_0)$ is retract
$\overline{k}$-rational. This implies $B_0(G)=0$ by Lemma
\ref{lemSa}. A contradiction.

In case $2^6\mid n$, the proof is similar by applying Theorem
\ref{thSB}.

\bigskip
\begin{remark}
In the proof of \cite[Lemma 5.6, page 478]{Bo}, Bogomolov tried to
prove that there do not exist $p$-groups $G$ of order $p^5$ with
$B_0(G)\neq 0$. He assumed that the commutator group $[G,G]$ was
abelian and discussed three situations when the order of $G/[G,G]$
was $p^2$, $p^3$, or $ \ge p^4$ (in general, if $G$ is a
non-abelian group of order $p^5$, then $[G,G]$ is abelian, since
$G$ has an abelian normal subgroup of order $p^3$ by a theorem of
Burnside). The case when $G/[G,G]=p^2$ was reduced to \cite[Lemma
4.11, page 478]{Bo} (see the first part of Theorem \ref{thmain}).
But this lemma is disproved by Theorem \ref{thmain}.
\end{remark}

\bigskip

\section{More groups $G$ with $B_0(G)\neq 0$}\label{seMore}

First we recall the classification of all groups of order $p^5$ up
to isomorphism. A list of groups of order $2^5$ (resp. $3^5$,
$5^5$, $7^5$) can be found in the database of GAP. However the
classification of groups of order $p^5$ dated back to Bagnera
(1898), Bender (1927), R. James (1980), etc. \cite{Ba, Be, Ja},
although some minor errors might occur in the classification
results finished before the computer-aided time. For examples, in
Bender's classification of groups of order $3^5$, one group is
missing, i.e. the group $\Delta_{10}(2111)a_2$ which was pointed
by \cite[page 613]{Ja}. A beautiful formula for the total number
of the groups of order $p^5$, for $p\geq 3$, was found by Bagnera
\cite{Ba} as
\[
2p+61+{\rm gcd}\{4,p-1\}+2{\rm gcd}\{3,p-1\}.
\]

A similar formula for groups $G$ of order $p^5$ $(p\geq 5)$ with
$B_0(G)\neq 0$ seems possible (see Theorem \ref{thB0n} and the
remark after it).

Before we state our results for $p\geq 5$, let us record the
result for $p=3$ first.
\begin{theorem}[Moravec \cite{Mo}]
Let $G$ be a group of order $3^5$.
Then $B_0(G)\neq 0$ if and only if $G=G(3^5,i)$ where $28\leq i\leq 30$.
\end{theorem}
\begin{proof}
Because Moravec proved this theorem by computers,
we will give a bare-hand proof of it.

For $G=G(3^5,28)$, the proof that $B_0(G)\neq 0$ is given in the
proof of Theorem \ref{thmain} in Section \ref{seproof} by using
Lemma \ref{lemf}.

If $G=G(3^5,29)$ or $G(3^5,30)$, the generators and relations of
$G$ are almost the same as $G(3^5,28)$, which are given  in the
proof of Theorem \ref{thmain} except the following relations for
$G(3^5,28)$
\[
f_1^3=f_4^3=f_5^3=1,\, f_2^3=f_4^{-1},\, f_3^3=f_5^{-1}
\]
should be replaced by
\[
f_4^3=f_5^3=1,\, f_1^3=f_5,\, f_2^3=f_4^{-1},\, f_3^3=f_5^{-1}
\]
and
\[
f_4^3=f_5^3=1,\, f_1^3=f_3^3=f_5^{-1},\, f_2^3=f_4^{-1}
\]
for $G(3^5,29)$ and $G(3^5,30)$ respectively.

Since $G(3^5,29)$ and $G(3^5,30)$ satisfy the conditions of Lemma
\ref{lemf}, the proof in Theorem \ref{thmain} for $G(3^5,28)$
works for $G(3^5,29)$ and $G(3^5,30)$ as well. Hence $B_0(G)\neq
0$ for these groups.

If $G$ is not isomorphic to $G(3^5,i)$ for $28\leq i\leq 30$, then
$\bC(G)$ is $\bC$-rational by \cite{CHK}.
Hence $B_0(G)=0$ by Lemma \ref{lemSa}.
\end{proof}
\begin{example}
We may try imitating the construction of the groups $G(3^5,i)$ where
$28\leq i\leq 30$ to define groups of order $p^5$ with $p\geq 5$.
Explicitly, take the $G(3^5,28)$ and define a group $G(28|p)$ by

$G(28|p)=\langle f_i : 1\leq i\leq 5\rangle$, $Z(G(28|p))=\langle f_5\rangle$,

$f_1^p=f_4^p=f_5^p=1$, $f_2^p=f_4^{-1}$, $f_3^p=f_5^{-1}$,

$[f_2,f_1]=f_3$, $[f_3,f_1]=f_4$, $[f_4,f_1]=[f_3,f_2]=f_5$, $[f_4,f_2]=[f_4,f_3]=1$.

When $p=3$, $G(28|p)$ is just the group $G(3^5,28)$.

What happens if $p=5,7,11,\ldots$ etc. ? Well, as an experiment,
using GAP to compute the group $G(28|5)$ and $G(28|7)$ first, we
find that they are groups of order $5^4$ and $7^4$ respectively,
instead of $5^5$ and $7^5$! Thus the group $G(28|p)$ for $p\geq 5$
is not a group of order $p^5$ at least when $p=5$ or $p=7$. Thus
we cannot apply Lemma \ref{lemf} to these groups.

Here is a proof for any prime number $p \ge 5$ that the three
groups $G(28|p)$, $G(29|p)$ and $G(30|p)$ are not of order $p^5$.

In Bender's classification \cite{Be}, the groups $G(3^5,28)$ and
$G(3^5,30)$ are just the groups $50$ for $m=1$ and $m=2$
respectively \cite[papge 70]{Be}. Note that Bender's group 50 is
defined only for $p=3$; there is no analogous groups for $p\geq 5$
in Bender's list. Hence $G(28|p)$ and $G(30|p)$ for $p\geq 5$ are
not groups of order $p^5$.

As to the group $G(3^5,29)$, it is the group overlooked by Bender
and was pointed out by James \cite[page 613]{Ja}. Thus the group
$G(29|p)$ does not appear in the list of Bender's classification.
But this group is James's $\Delta_{10}(2111)a_2$ \cite[page
621]{Ja}, while the groups $\Delta_{10}(2111)a_1$ and
$\Delta_{10}(2111)a_3$ are the GAP groups $G(3^5,28)$ and
$G(3^5,30)$ respectively (recall James's notation: when $p=3$, the
group $\Phi_{10}(2111)a_r$ is written as $\Delta_{10}(2111)a_r$).
Note that, for $p \ge 5$, the group corresponding to $G(3^5,29)$
is defined by a modified way in \cite[page 621]{Ja}, i.e. the
definitions of the corresponding group and the group $G(29|p)$ are
different.
\end{example}

\bigskip
Now we turn to the case $p\geq 5$.

\begin{definition}
Let $p\geq 3$ and define a group $G(1|p)$ by

$G(1|p)=\langle f_i : 1\leq i\leq 5\rangle$, $Z(G(1|p))=\langle f_5\rangle$,

$f_i^p=1$ for $1\leq i\leq 5$,

$[f_2,f_1]=f_3$, $[f_3,f_1]=f_4$, $[f_4,f_1]=[f_3,f_2]=f_5$, $[f_4,f_2]=f_4,f_3]=1$.
\end{definition}
When $p\geq 5$, the group $G(1|p)$ is the group 42 in \cite[page
69]{Be} and $\Phi_{10}(1^5)$ in \cite[page 621]{Ja}. Note that, in
\cite[page 69]{Be}, it is emphasized that the group 42 is defined
only when $p\geq 5$. In summary, if $p\geq 5$, the group $G(1|p)$
is a group of order $p^5$.

When $p=3$, the group $G(1|p)$ is not a group of order $3^5$
either by computing with the aid of GAP or by the classification
of Bender and James \cite{Be, Ja}.
\begin{definition}
Let $p\geq 3$ and $\alpha$ be the smallest positive integer which
is a primitive root $\pmod{p}$. Define $c_2={\rm gcd}\{4,p-1\}-1$.
For $0\leq r\leq c_2$, we define a group $G_r(2|p)$ by

$G_r(2|p)=\langle f_i : 1\leq i\leq 5\rangle$, $Z(G_r(2|p))=\langle f_5\rangle$,

$f_1^p=f_5^{\alpha^r}, f_i^p=1$ for $2\leq i\leq 5$,

$[f_2,f_1]=f_3$, $[f_3,f_1]=f_4$, $[f_4,f_1]=[f_3,f_2]=f_5$,
$[f_4,f_2]=f_4,f_3]=1$.
\end{definition}
When $p\geq 5$, the groups $G_r(2|p)$ for $0\leq r\leq c_2$ are the group 43
in \cite[page 69]{Be}, and are the groups $\Phi_{10}(2111)a_r$ in
\cite[page 621]{Ja}.
When $p\geq 5$, these groups are groups of order $p^5$.
When the parameters $(p,r)$ where $p\geq 5$, $0\leq r\leq c_2$ are distinct,
the corresponding groups $G_r(2|p)$ are not isomorphic to each other.
In conclusion, if $p\geq 5$, the groups $G_r(2|p)$ where $0\leq r\leq c_2$
are non-isomorphic groups of order $p^5$.

Similarly, when $p=3$, the groups $G_0(2|3)$ and $G_1(2|3)$ are not
groups of order $3^5$.
\begin{definition}
Let $p\geq 3$ and $\alpha$ be the smallest positive integer which
is a primitive root $\pmod{p}$. Define $c_3={\rm gcd}\{3,p-1\}-1$.
For $0\leq r\leq c_3$, we define a group $G_r(3|p)$ by

$G_r(3|p)=\langle f_i : 1\leq i\leq 5\rangle$, $Z(G_r(3|p))=\langle f_5\rangle$,

$f_2^p=f_5^{\alpha^r}$, $f_1^p=f_i^p=1$ for $2\leq i\leq 5$,

$[f_2,f_1]=f_3$, $[f_3,f_1]=f_4$, $[f_4,f_1]=[f_3,f_2]=f_5$,
$[f_4,f_2]=[f_4,f_3]=1$.
\end{definition}
When $p\geq 5$, the groups $G_r(3|p)$ for $0\leq r\leq c_3$ are
the group $52$ in \cite[page 70]{Be}, and are the groups
$\Phi_{10}(2111)b_r$ in \cite[page 621]{Ja}. When $p\geq 5$, the
groups $G_r(3|p)$ where $0\leq r\leq c_3$ are non-isomorphic
groups of order $p^5$. Similarly, when $p=3$, the group $G_0(3|3)$
is not a group of order $3^5$.
\begin{theorem}\label{thB0n}
Let $p\geq 5$ and $G$ be a group of order $p^5$ isomorphic to any
one of $G(1|p)$, $G_r(2|p)$ for $0\leq r\leq c_2$, $G_r(3|p)$ for
$0\leq r\leq c_3$. Then $B_0(G)\neq 0$. The total number of such
groups is
\[
1+{\rm gcd}\{4,p-1\}+{\rm gcd}\{3,p-1\}.
\]
\end{theorem}
\begin{proof}
$G$ satisfies the conditions of Lemma \ref{lemf}. Apply Lemma
\ref{lemf}. We find $B_0(G)\neq 0$. The total number of such
groups is $1+c_2+c_3$, which is the same as the formula given in
the theorem.
\end{proof}

\begin{remark}
According to the notation of \cite{Ja}, all the groups in Theorem
\ref{thB0n} consist of a single isoclinism class among groups of
order $p^5$.

We don't know whether these groups are the only groups $G$ of
order $p^5$ (for $p\ge 5$) with nontrivial $B_0(G)$. Moravec
informed us recently it was the case when $p=5$ and $7$ (besides
$p=3$) by using his computer program. But we don't have a
theoretic proof for it. Of course, the above computation of
Moravec may be extended to the cases $p=11, 13$, etc. if more
powerful computers are used. At present, it requires a lot of
computer time to check this question even when $p = 11$.

For the convenience of readers, the following list provides the
GAP code
numbers of the groups  $G(1|p)$, $G_r(2|p)$, $G_r(3|p)$ when $5\leq p\leq 37$. \\

\begin{tabular}{l|l|l|l|}
$p$ & $G(1|p)$ & $G_r(2|p)$ & $G_r(3|p)$\\\hline
$5$ & $G(33)$ & $G(34),G(35),G(36),G(37)$ & $G(38)$\\
$7$ & $G(37)$ & $G(41),G(42)$ & $G(38),G(39),G(40)$\\
$11$ & $G(39)$ & $G(41),G(42)$ & $G(40)$\\
$13$ & $G(43)$ & $G(47),G(48),G(49),G(50)$ & $G(44),G(45),G(46)$\\
$17$ & $G(45)$ & $G(47),G(48),G(49),G(50)$ & $G(46)$\\
$19$ & $G(49)$ & $G(53),G(54)$ & $G(50),G(51),G(52)$\\
$23$ & $G(51)$ & $G(53),G(54)$ & $G(52)$\\
$29$ & $G(57)$ & $G(59),G(60),G(61),G(62)$ & $G(58)$\\
$31$ & $G(61)$ & $G(65),G(66)$ & $G(62),G(63),G(64)$\\
$37$ & $G(67)$ & $G(71),G(72),G(73),G(74)$ & $G(68),G(69),G(70)$
\end{tabular}
\\~\\

(*) In the above list, note that the number $p^5$ is suppressed
from the GAP code number $G(p^5,i)$. For example, the group
$G(47)$ in the row corresponding to $p=13$ denotes the group
$G(13^5,47)$, while the group $G(47)$ in the row corresponding to
$p=17$ denotes the group $G(17^5,47)$.
\end{remark}

\bigskip

\end{document}